\documentclass[a4paper,10pt]{article}

\usepackage{amsfonts,amssymb,amsmath,xypic}

\numberwithin{equation}{section}

\usepackage{ulem}

\usepackage{tikz}
\usepackage{pifont}
\usepackage{tikz-cd}

\newtheorem{theo}{Theorem}[section]
\newtheorem{coro}[theo]{Corollary}
\newtheorem{lemm}[theo]{Lemma}
\newtheorem{prop}[theo]{Proposition}
\newtheorem{defi}[theo]{Definition}
\newtheorem{rema}[theo]{Remark}
\newtheorem{exam}[theo]{Example}

\newenvironment{proof}{\noindent \textbf{{Proof.}} \sf}

\def\qed{\hfill $\diamond$ \bigskip}

\def\C{{\mathcal C}}

\def\H{{\mathcal H}}

\def\lim{\mathop{\rm lim}\nolimits}

\def\Ext{\mathsf{Ext}}
\def\Hom{\mathsf{Hom}}
\def\Tor{\mathsf{Tor}}

\def\Ker{\mathsf{Ker}}

\def\Coker{\mathsf{Coker}}

\def\Im{\mathsf{Im}}

\def\dim{\mathsf{dim}}

\def\la{\Lambda}

\begin{document}
\sf

\title{Strongly stratifying ideals, Morita contexts and Hochschild homology}
\author{Claude Cibils,  Marcelo Lanzilotta, Eduardo N. Marcos,\\and Andrea Solotar
\thanks{\footnotesize This work has been supported by the projects  UBACYT 20020170100613BA, PIP-CONICET 11220200101855CO, USP-COFECUB.
The third mentioned author was supported by the thematic project of FAPESP 2014/09310-5, research grants from CNPq 302003/2018-5 and  310651/2022-0  and acknowledges support from the ``Brazilian-French Network in Mathematics". The fourth mentioned author is a research member of CONICET (Argentina), Senior Associate at ICTP and visiting Professor at Guangdong Technion-Israel Institute of Technology.}}

\date{}
\maketitle
\begin{abstract}

We consider stratifying ideals of finite dimensional algebras in relation with Morita contexts. A Morita context is an algebra built on a data consisting of two algebras, two bimodules and two morphisms.  For a strongly stratifying Morita context - or equivalently for a strongly stratifying ideal - we show  that Han's conjecture holds if and only if it holds for the diagonal subalgebra. The main tool is the Jacobi-Zariski long exact sequence. One of the main consequences is that Han's conjecture holds for an algebra admitting a strongly (co-)stratifying chain whose steps verify Han's conjecture.

If Han's conjecture is true for local algebras and an algebra $\Lambda$ admits a primitive strongly (co-)stratifying chain, then Han's conjecture holds for $\Lambda$.

\end{abstract}




\section{\sf Introduction}

In this paper we consider finite dimensional associative algebras over a field $k$, which we call \textit{algebras} for short. We study Morita contexts  in relation with Han's conjecture, more details are given below. The intention of this paper is to reduce the analysis of Han's conjecture on an algebra to an easier algebra where Hochschild homology is more manageable.

 A Morita context \cite{BASS,JACOBSON} is  a matrix algebra built on two ``diagonal" algebras $A$ and $B$, two bimodules $M$ and $N$ and two bimodule maps $\alpha$ and $\beta$ verifying natural conditions equivalent to associativity of the product of the matrix algebra - see Definition \ref{def MC}.

One of the main results that this paper relies on is that algebras with a distinguished idempotent are essentially the same that Morita contexts, see for instance  \cite{BUCHWEITZ,GREEN PSAROUDAKIS}. For later use, we start Section \ref{clear} by recalling this well known fact in a categorical framework.

Stratifying ideals are defined by E. Cline, B. Parshall and L. Scott in \cite{CLINE PARSHALL SCOTT} and further studied in \cite{ANGELERI KOENIG LIU YANG,HAN2014,HERMANN2016}. We show that for Morita contexts the definition of a stratifying ideal  of an algebra  generated by an idempotent translates into the conditions $\Tor_n^A(M,N)=0$ for $n>0$ and $\beta$  injective. See \cite[Proposition 4.1]{GAO PSAROUDAKIS}.

Recollements are introduced by A. A. Beilinson, J. Bernstein and P. Deligne in \cite{BEILINSON BERNSTEIN DELIGNE}. Let  $\Lambda e\Lambda$ be a stratifying ideal of an algebra $\Lambda$, where $e$ is an idempotent of $\Lambda$. This gives us a recollement of the unbounded derived category $D(\Lambda)$ of complexes of $\Lambda$-modules relative to the unbounded derived categories $D(\Lambda/\Lambda e \Lambda)$ and $D(e\Lambda e)$. See \cite{CLINE PARSHALL SCOTT,HAN2014}.

In Section \ref{clear} we also introduce strongly stratifying ideals generated by an idempotent. For Morita contexts, this corresponds to $\Tor_n^A(M,N)=0$ for all $n$, and $\Tor_n^B(N,M)=0$ for $n>0$.

The Hochschild homology $HH_*(\Lambda)$ of an algebra $\Lambda$ (see \cite{HOCHSCHILD1945}, \cite{WEIBEL}, \cite{WITHERSPOON}) is called \textit{finite} if $HH_*(\Lambda)=0$ for $*>N$ for some $N$. A main purpose of this paper is to study under which circumstances the Hochschild homology of a Morita context being finite implies the Hochschild homology of its diagonal algebra is also finite. Our motivation is to attack Han's conjecture \cite{HAN2006}, which states that if an algebra has finite Hochschild homology, then it should have finite global dimension. For results in this direction, see for instance \cite{AVRAMOVVIGUE,  BERGHERDMANN, BERGHMADSEN2009, BERGHMADSEN2017, BACH, CLMS Pacific, CLMS2022bounded, CIBILSREDONDOSOLOTAR, SOLOTARSUAREZVIVAS, SOLOTARVIGUE}. Note that if Han's conjecture is true, then the following dichotomy holds: either $HH_*(\Lambda)$ is infinite or $HH_*(\Lambda)=0$ for $*>0$. See \cite{HAN2006}.

Previous results by  L. Angeleri H\"{u}gel, S. Koenig, Q. Liu and D. Yang in \cite{ANGELERI KOENIG LIU YANG 2017 472}  provides a motivation for approaching Han's conjecture through recollements. Roughly, the finiteness of the projective dimension is preserved for an idempotent which gives a stratifying ideal, through the recollement of the derived category. Moreover B. Keller in \cite{KELLER1998} and Y. Han in \cite{HAN2014} showed that in the same situation there is a long exact sequence relating the Hochschild homologies of the algebras involved.

In Section \ref{JJJZZ} we adjust and extend the Jacobi-Zariski long nearly exact sequence obtained in \cite{CLMS2021nearlyexactJZ}. The additional hypothesis for fitting  Theorem 4.2 of  \cite{CLMS2021nearlyexactJZ} is given in (\ref{missing}). This sequence links Hochschild homology of $\Lambda$ to the relative one with respect to a subalgebra introduced by G. Hochschild in \cite{HOCHSCHILD1956}. With this adjustement, we confirm our previous results in \cite{CLMS2021nearlyexactJZ} and the earlier Jacobi-Zariski long exact sequence of A. Kaygun in \cite{KAYGUN}.

Let  $\Lambda$ be an algebra with a distinguished idempotent $e$ satisfying that $\Lambda e \Lambda$ is a strongly stratifying ideal, and let $f=1-e$.
Using the previously  described tools, in Section \ref{HH of strongly str MC} we obtain the key result of this paper, that is Theorem \ref{key}: if $\Lambda$ has finite Hochschild homology, then the same holds for $e\Lambda e$ and $f\Lambda f$. In Section \ref{HAN}, we prove our main result:  $\Lambda$ verifies  Han's conjecture if and only if its subalgebra $e\Lambda e\times f\Lambda f$ does.

We next consider algebras admitting a strongly stratifying chain, that is those algebras  with an ordered complete system of orthogonal idempotents such that the successive quotients of the induced filtration by ideals are strongly stratifying in the corresponding algebra. We obtain the following interesting consequence of our key result Theorem \ref{key}.  Let $\C$ be a class of algebras verifying Han's conjecture which is closed by taking quotients. If an algebra $\la$ admits a strongly stratifying chain $\{e_1,\dots, e_n\}$ such that all the algebras $e_i\Lambda e_i$ belong to $\C$, then Han's conjecture is true for $\la$.

To avoid classes of algebras closed by taking quotients, we filter instead an algebra $\la$ by algebras $f\la f$ where the $f$'s are partial decreasing sums of a complete system of orthogonal idempotents $\{e_1,\dots, e_n\}$. We consider strongly co-stratifying chains, that is the $f$'s provide ideals which are strongly stratifying in the next algebra.  We infer from Theorem \ref{key} another main result: if an algebra $\la$ admits a strongly co-stratifying chain $\{e_1,\dots, e_n\}$ such that all the algebras $e_i\Lambda e_i$ verify Han's conjecture, then Han's conjecture is true for $\la$.

In particular if Han's conjecture is true for local algebras and if an algebra admits a primitive strongly (co-)stratifying chain, then Han's conjecture is true for this algebra.

Those algebras admitting a strongly stratifying or co-stratifying chain will be compared with standardly stratified algebras (see for instance \cite{AGOSTON DLAB LUKACS, AGOSTON HAPPEL LUKACS UNGER, MARCOS MENDOZA SAENZ SANTIAGO, XI}) in a forthcoming paper.

In the last section, assuming \textit{ad-hoc} projectivity conditions on the bimodules of a Morita context, we give patterns to produce examples.

\noindent\textit{Acknowledgements:} We thank Octavio Mendoza for a comment which led to the definition of algebras admitting a  strongly (co-)stratifying chain and ultimately to the mentioned result for these algebras.

\section{\sf Stratifying and strongly stratifying Morita contexts}\label{clear}

Let $k$ be a field. As mentioned,  a finite dimensional associative $k$-algebra is called an \textit{algebra} throughout this paper. We first recall the definition of Morita context and show that it is the same that an algebra with a distinguished  idempotent $e$. Then we specialize to the case where the two-sided ideal generated by $e$ is stratifying, in order to obtain a stratifying Morita context.

As E.L. Green and C. Psaroudakis pointed out in \cite{GREEN PSAROUDAKIS}, Morita contexts have been introduced by H. Bass \cite{BASS} in 1962, and considered by P. M. Cohn \cite{COHN} in 1996.

\begin{defi}\label{def MC}
A \textit{Morita context} is a matrix algebra $\begin{pmatrix}
                                              A  & N \\
                                              M & B
                                            \end{pmatrix}_{\alpha, \beta}$
where $A$ and $B$ are algebras, $M$ and $N$ are finite dimensional $B-A$ and  $A-B$-bimodules respectively, and $\alpha$ and $\beta$ are $A$ and $B$-bimodule maps respectively
$$\alpha : N\otimes_B M \to A  \mbox{ and } \beta : M\otimes_A N \to B$$
which verify  ``associativity" conditions
\begin{align}\label{associativity conditions}
  \alpha(n\otimes m) n' = n \beta(m\otimes n') \mbox{ and } \beta (m\otimes n) m' = m\alpha (n\otimes m').
\end{align}
The product of the Morita context is the matrix product using the above, namely
$$\begin{pmatrix}
  a & n \\
  m & b
\end{pmatrix}\begin{pmatrix}
  a' & n' \\
  m' & b'
\end{pmatrix}=\begin{pmatrix}
                aa'+\alpha(n\otimes m') & an'+nb' \\
                ma'+bm' & \beta(m\otimes n')+ bb'
              \end{pmatrix}.$$
This product is associative if and only if the ``associativity" conditions on $\alpha$ and $\beta$ hold.
\end{defi}

For completeness, we recall the well known fact that Morita contexts are in bijection with $k$-categories with an ordered pair of objects and morphisms are $k$-vector spaces. Indeed, starting with a Morita context, the associated category has $A$ (resp. $B$) as endomorphism algebra of the first (resp. second) object; morphisms from the first (resp. second) object to the second (resp. first) object are $M$ (resp. $N$); finally compositions of those morphisms are given by $\alpha$ and $\beta$, their ``associativity" conditions ensure that the composition is associative. Conversely, given a $k$-category with an ordered pair of objects, the Morita context has the algebras of endomorphisms of the objects on its diagonal. On the antidiagonal, the bimodules are the morphisms  between the objects - which are indeed bimodules over the previous algebras of endomorphisms. The maps $\alpha$ and $\beta$ are provided by the composition of the category.

\begin{defi}
The objects of the category $\mathsf{Morita.Contexts}$ are the Morita contexts. A morphism of  Morita contexts $$\begin{pmatrix}
                                              A  & N \\
                                              M & B
                                            \end{pmatrix}_{\alpha, \beta} \to \begin{pmatrix}
                                              A'  & N' \\
                                              M' & B'
                                            \end{pmatrix}_{\alpha', \beta'}$$
is a quadruple of maps
$$(A\stackrel{\varphi}{\to} A', B\stackrel{\psi}{\to} B', M\stackrel{f}{\to} M', N\stackrel{g}{\to}N'),$$
where $\varphi$ and $\psi$ are algebra maps - they provide $M'$ and $N'$ with respective structures of $B-A$ and $A-B$-bimodule. Moreover $f$ and $g$ are respectively $B-A$ and $A-B$-bimodule maps. In addition, these maps verify the following conditions:
$$\varphi\left(\alpha(n\otimes m)\right)=\alpha'\left(g(n)\otimes f(m)\right) \mbox{ and }
\psi\left(\beta(m\otimes n)\right)= \beta'\left(f(m)\otimes g(n)\right).$$
\end{defi}
The direct sum of the four maps of a morphism is an algebra map.

Note that a morphism between Morita contexts is equivalent to a functor between the corresponding categories with an ordered pair of objects, which respects the ordered pairs.

On the other hand the category of $k$-algebras with an idempotent is as follows.

\begin{defi}
The objects of the category $\mathsf{Algebras.Idempotent}$ are pairs $(\Lambda, e)$ where $\Lambda$ is a $k$-algebra and $e$ is a \textit{distinguished} idempotent of $\Lambda$. A morphism $\varphi: (\Lambda, e) \to (\Lambda', e')$ is a morphism of algebras $\varphi: \Lambda \to \Lambda'$ such that $\varphi(e)=e'$.
\end{defi}

The following result is well known.

\begin{theo}\label{identical}
The categories $\mathsf{Morita.Contexts}$ and $\mathsf{Algebras.Idempotent}$ are isomorphic.
\end{theo}
\begin{proof}
Let $\begin{pmatrix}
                                              A  & N \\
                                              M & B
                                            \end{pmatrix}_{\alpha, \beta}$
be a Morita context. The associated object in \\$\mathsf{Algebras.Idempotent}$ is the Morita context with distinguished idempotent  $\begin{pmatrix}
                                              1_A  & 0 \\
                                              0 & 0
                                            \end{pmatrix}.$ Starting from a morphism of Morita contexts, that is a quadruple of appropriate maps, their direct sum clearly preserves the distinguished idempotents.

Conversely, let $(\Lambda, e)$ be an object in $\mathsf{Algebras.Idempotent}$ and consider the idempotent $f=1-e$. Then $\begin{pmatrix}
                                              e\Lambda e  & e\Lambda f \\
                                              f\Lambda e & f \Lambda f
                                            \end{pmatrix}_{\alpha, \beta}$ is a Morita context
where
$$\alpha : e\Lambda f\otimes_{f \Lambda f} f\Lambda e \to e\Lambda e \mbox{ and } \beta :   f\Lambda e\otimes_{e \Lambda e} e\Lambda f \to f\Lambda f $$ are given by the product of $\Lambda$. Observe that a morphism $\varphi:(\Lambda, e)\to (\Lambda', e')$ also verifies $\varphi(f)=f'$, where $f'=1-e'$. Therefore, a morphism of algebras with distinguished idempotents provides a morphism of the corresponding Morita contexts.
These functors are mutual inverses.\qed
  \end{proof}

\begin{rema}
The previous Theorem \ref{identical} generalises to an algebra with a finite complete set  of  orthogonal idempotents (non necessarily primitive), see for instance \cite{CLMS2019,CIBILSMARCOS2006}.
\end{rema}

In 1996 E. Cline, B. Parshall and L. Scott \cite{CLINE PARSHALL SCOTT} considered a stratifying ideal generated by an idempotent of an algebra, we next recall its definition.

\begin{defi}\label{idempotent s}\cite{CLINE PARSHALL SCOTT}
Let $\Lambda$ be an algebra and $e\in\Lambda$ an idempotent. The ideal $\Lambda e \Lambda$ is stratifying if
\begin{enumerate}
  \item $\Tor_n^{e\Lambda e}(\Lambda e, e\Lambda)=0$ for $n>0$
  \item The surjection given by the product $\Lambda e\otimes_{e\Lambda e} e\Lambda \to \Lambda e \Lambda$ is injective.
\end{enumerate}
\end{defi}

\begin{rema} (\cite{CLINE PARSHALL SCOTT},\cite[Example 1, p.537]{HAN2014}) Let $\Lambda$ be an algebra and let $D(\Lambda)$ denote the unbounded derived category of complexes of left $\Lambda$-modules. Let $e\in \Lambda$ be an idempotent such that $\Lambda e \Lambda$ is a stratifying ideal.  Then $D(\Lambda)$ admits a recollement relative to $D(\Lambda/\Lambda e \Lambda)$ and $D(e\Lambda e)$, which is interpreted  as a short exact sequence of triangulated categories, as introduced by A. A. Beilinson, J. Bernstein and P. Deligne \cite{BEILINSON BERNSTEIN DELIGNE} in 1982.
\end{rema}

In 2009 S. Koenig and H. Nagase  \cite[p. 888]{KOENIG NAGASE}  showed that an idempotent $e\in \Lambda$ gives a stratifying ideal if and only if the canonical surjection $\Lambda \to \Lambda/\Lambda e \Lambda$ induces isomorphisms
$$\Ext^*_{\Lambda/\Lambda e \Lambda} (X,Y)\to \Ext^*_{\Lambda} (X,Y)$$
for all $\Lambda/\Lambda e \Lambda$-modules $X$ and $Y$.

It is worth noting that the above is precisely the definition of a \textit{strong} idempotent ideal given in 1992 by M. Auslander, M.I. Platzeck and G. Todorov, see \cite[p. 669]{AUSLANDER PLATZECK TODOROV}. They are not to be confused with ``strongly stratifying" ideals that we will consider later.

The following definition will allow to consider stratifying ideals in the framework of Morita contexts.

\begin{defi}
A Morita context $\begin{pmatrix}
                                              A  & N \\
                                              M & B
                                            \end{pmatrix}_{\alpha, \beta}$
is \textit{stratifying} if \begin{enumerate}
  \item $\Tor_n^A(M,N)=0$ for $n>0,$
  \item $\beta$ is injective.
\end{enumerate}

  \end{defi}

 Both stratifying definitions agree through the identification of Theorem \ref{identical} between algebras with distinguished idempotents and Morita contexts:
\begin{theo}\cite[Proposition 4.1]{GAO PSAROUDAKIS}\label{s iff s}
Let $\Lambda =\begin{pmatrix}
                                              A  & N \\
                                              M & B
                                            \end{pmatrix}_{\alpha, \beta}$ be a Morita context and let $e$ be the idempotent $\begin{pmatrix}
                                              1  & 0 \\
                                              0 & 0
                                            \end{pmatrix}$. The Morita context is stratifying if and only
the ideal $\Lambda e \Lambda $ is stratifying.

\end{theo}

\begin{proof}
Let $\Lambda = \begin{pmatrix}
                                              A  & N \\
                                              M & B
                                            \end{pmatrix}$ and $e=\begin{pmatrix}
                                              1 & 0 \\
                                              0 & 0
                                            \end{pmatrix}$. We have
$$\Lambda e = A\oplus M \hskip5mm e\Lambda = A\oplus N$$
respectively as right and left $A$-modules because $e\Lambda e= A$. Therefore
$$\Tor_n^{e\Lambda e}(\Lambda e, e\Lambda)= \Tor_n^A(A , A)\oplus \Tor_n^A( A,N )\oplus \Tor_n^A(M ,A )\oplus \Tor_n^A( M,N ).$$
The three first direct summands on the right side of the equality are $0$ for $n>0$. This shows the equivalence on the Tor conditions.

Then note that $$\Lambda e \Lambda = \begin{pmatrix}
                                              A  & N \\
                                              M & \Im\beta
                                            \end{pmatrix}.$$
Moreover the morphism given by the product $\Lambda e\otimes_{e\Lambda e} e\Lambda \to \Lambda e \Lambda$ decomposes diagonally as the direct sum of four morphisms
$$A\otimes_A A \to A, \ \ A\otimes_A N\to N, \ \ M\otimes_A A\to M\ \mbox{ and }\ M\otimes_A N \to \Im\beta.$$
The first three are clearly isomorphisms, while the last one provides the equivalence on the injectivity conditions. \qed
\end{proof}
\begin{rema}\label{Lambda modulo the ideal through e}
From the proof of the previous Theorem \ref{s iff s} we have
\begin{equation*}
\Lambda/\Lambda e \Lambda = B/\Im\beta.
\end{equation*}
\end{rema}

We will consider strongly stratifying ideals as follows.
\begin{defi} \label{exigent stratifying ideal}\label{strongly stratifying ideal} Let $(\Lambda, e)$ be an algebra with a distinguished idempotent $e$, and let $f=1-e$. The ideal $\Lambda e \Lambda$  is \textit{strongly stratifying} if

 \begin{enumerate}
  \item $\Tor_n^{e\Lambda e}(\Lambda e, e\Lambda)=0$ for $n>0,$
  \item \label{two} $f\Lambda e \otimes_{e\Lambda e} e\Lambda f =0,$
  \item \label{three} $\Tor_n^{f\Lambda f}(\Lambda f, f\Lambda)=0$ for $n>0.$
\end{enumerate}
\end{defi}

\begin{rema}\
\begin{itemize}
  \item In Proposition \ref{idempotent es implies s} we will prove that if $\Lambda e \Lambda$  is strongly stratifying, then $\Lambda e \Lambda$ is indeed stratifying.
  \item The last requirement of Definition \ref{strongly stratifying ideal} for $\Lambda e \Lambda$ to be a strongly stratifying ideal coincides with the  first requirement of the definition for $\Lambda f\Lambda$ to be a stratifying ideal (see Definition \ref{idempotent s}).
\end{itemize}

\end{rema}

\begin{exam}\label{LIU}
We consider the example of \cite[Example 4.4]{LIU VITORIA YANG} and  \cite[Example 2.3]{ANGELERI KOENIG LIU YANG}. Let $\Lambda$ be the following bound quiver algebra

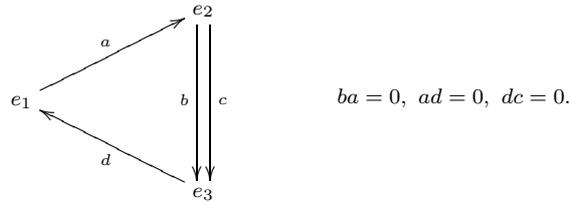
\footnotesize
\[\xymatrix{&&e_2\ar@<.7ex>[dd]^{c}\ar@<-.7ex>[dd]_{b}\\
e_1\ar[rru]^{a}&&&\\
&& e_3\ar[llu]^{d}}\hspace{10pt}\xymatrix{\\ ba=0,~ ad=0,~dc=0.}\]
\normalsize
As proved in \cite{LIU VITORIA YANG}, the idempotent $e=e_2+e_3$ provides a stratifying ideal $\Lambda e \Lambda$. We assert that  $\Lambda e \Lambda$ is a strongly stratifying ideal.

The algebra $A=e\Lambda e$ is the Kronecker algebra $\begin{tikzcd}
	{e_2} \\
	{e_3}
	\arrow["c", shift left=1, from=1-1, to=2-1]
	\arrow["b"', shift right=1, from=1-1, to=2-1]
\end{tikzcd} $.

Let $f=e_1=1-e$. Consider the right $A$-module $M=f\Lambda e$ and the left $A$-module $N= e\Lambda f$. Their associated quiver representations are
$$\begin{tikzcd}
	Me_2=k\{db\} \\
	Me_3=k\{d\}
	\arrow["1", shift left=1, from=2-1, to=1-1]
	\arrow["0"', shift right=1, from=2-1, to=1-1]
\end{tikzcd}
\hskip2cm
\begin{tikzcd}
	e_2N=k\{a\} \\
	e_3N=k\{ca\}
	\arrow["1", shift left=1, from=1-1, to=2-1]
	\arrow["0"', shift right=1, from=1-1, to=2-1]
\end{tikzcd}$$
We assert that $M\otimes_A N=0$.
Indeed,
\begin{align*}
&db\otimes a=d\otimes ba= d\otimes 0 =0, \ \ \ db\otimes ca = dbe_2\otimes ca= db\otimes e_2ca =db\otimes 0= 0\\
&d\otimes a=de_3\otimes a=d\otimes e_3a=d\otimes 0=0, \  \ \ d\otimes ca = dc\otimes a = 0\otimes a =0.
\end{align*}

Moreover $B=f\Lambda f =k$, thus the last requirement of Definition \ref{strongly stratifying ideal} holds.
For the sake of completeness we next check that $\Tor^A_n( M,N)=0$ for $n>0$, after \cite{LIU VITORIA YANG}. Consider the projective left $A$-modules $$P_1=\begin{tikzcd}
	0 \\
\\
	{k\{e_3\}}
	\arrow[shift right=1, from=1-1, to=3-1]
	\arrow[shift left=1, from=1-1, to=3-1]
\end{tikzcd}\hskip2cm
P_0=\begin{tikzcd}
	{k\{e_1\}} \\
\\
	{k\{b\}\oplus k\{c\}}
	\arrow["\begin{pmatrix}
	          0 \\
	          1
	        \end{pmatrix}", shift left=1, from=1-1, to=3-1]
	\arrow["\begin{pmatrix}
	          1 \\
	          0
	        \end{pmatrix}"', shift right=1, from=1-1, to=3-1]
\end{tikzcd}
$$
and the projective resolution of $N$
$$0\to P_1\to P_0\to N\to 0$$
where the map $P_1\to P_0$ sends $e_3$ to $b$. It is straightforward to verify that $M\otimes_AP_1= k\{d\otimes e_3\}$ and $M\otimes_A P_0= k\{d\otimes b\}$. Hence the complex computing $\Tor^A_n(M,N)$ for $n\geq 0$ is $$0\to M\otimes_AP_1\to M\otimes_A P_0\to 0$$
which is exact. Note that this gives another proof that $M\otimes_A N=0$.
\end{exam}
\begin{defi}\label{exigent stratifying MC}
A  Morita context $\begin{pmatrix}
                                              A  & N \\
                                              M & B
                                            \end{pmatrix}_{\alpha, \beta}$ is \textit{strongly stratifying} if
                                            \begin{enumerate}\label{exigent stratifying MC explicit}\label{strongly stratifying MC explicit}
                                         \item     $\Tor_n^A(M,N)=0 \text{ for } n>0,$
                                         \item $M\otimes_A N=0,$
                                        \item      $\Tor_n^B(N,M)=0 \text{ for } n>0.$
                                            \end{enumerate}
\end{defi}

\begin{rema}\label{esmc is smc}
The morphism $\beta$ of a strongly stratifying Morita context is $\beta: 0\to B$ which is of course injective. Therefore strongly stratifying Morita contexts are indeed stratifying.
\end{rema}

\begin{prop}\label{es iff es}
Let $\Lambda$ be an algebra with a distinguished idempotent $e$, and let $f=1-e$. The associated Morita context   $\begin{pmatrix}
                                              e\Lambda e  & e\Lambda f \\
                                              f\Lambda e& f\Lambda f
                                            \end{pmatrix}_{\alpha, \beta}$ is strongly stratifying if and only if the ideal $\Lambda e \Lambda$  is strongly stratifying.

\end{prop}

\begin{proof} We have
\begin{align*}
  \Tor_n^{e\Lambda e} (\Lambda e , e\Lambda) = &\Tor_n^{e\Lambda e} (e\Lambda e , e\Lambda e)\ \oplus\  \Tor_n^{e\Lambda e} (e\Lambda e , e\Lambda f)\ \oplus \\
  &\Tor_n^{e\Lambda e} (f\Lambda e , e\Lambda e)\ \oplus\ \Tor_n^{e\Lambda e} (f\Lambda e , e\Lambda f).
\end{align*}
Of course $e\Lambda e$  is a projective left $e\Lambda e$-module, it is also projective as a right $e\Lambda e$-module. Thus for $n>0$
$$\Tor_n^{e\Lambda e} (\Lambda e , e\Lambda) = \Tor_n^{e\Lambda e} (f\Lambda e , e\Lambda f).$$
Moreover for $n=0$ the second condition in both definitions  is $f\Lambda e\otimes_{e\Lambda e} e\Lambda f =0$.

Analogously, for $n>0$ we have
 $$\Tor_n^{f\Lambda f} (\Lambda f , f\Lambda) = \Tor_n^{f\Lambda f} (e\Lambda f , f\Lambda e).$$
 \qed
\end{proof}
\begin{prop} \label{idempotent es implies s}
Let $\Lambda$ be an algebra and $e\in\Lambda$ an idempotent. If the ideal $\Lambda e \Lambda$ is strongly stratifying, then it is stratifying.
\end{prop}
\begin{proof}
By Proposition \ref{es iff es}, the Morita context is strongly stratifying. Thus the Morita context is stratifying by  Remark \ref{esmc is smc}. Finally the ideal $\Lambda e \Lambda$ is stratifying by  Theorem \ref{s iff s}.\qed
\end{proof}

\section{\sf Jacobi-Zariski long nearly exact sequence and bounded extensions of algebras}\label{JJJZZ}

We begin this section with a brief account of relative Hochschild homology with respect to a subalgebra, and not with respect to an ideal, as it is often considered. Next, we will provide the adjusted and extended version of the Jacobi-Zariski long near exact sequence of  in \cite{CLMS2021nearlyexactJZ}, then we confirm the results of \cite{CLMS2022bounded}.

Let $C\subset \Lambda$ be an extension of algebras, that is $C$ is a subalgebra of $\Lambda$. The relative projective $\Lambda$-modules are direct summands of $\Lambda\otimes_C V$, where $V$ is any left $C$-module. Note that if $C=k$, the relative projectives are the usual $\Lambda$-projective modules.

Recall that a relative projective resolution of a $\Lambda$-module $U$ requires the existence of a $C$-contracting homotopy, by  \cite[p. 250]{HOCHSCHILD1956} it always exits. For a right $\Lambda$-module $U'$, this leads to well defined vector spaces $\Tor_*^{\Lambda|C}(U', U)$.

Let $C^\mathsf{e}$ be the enveloping algebra $C\otimes_k C^{\mathsf{op}}$ of an algebra, and consider the extension of algebras $C^\mathsf{e}\subset \Lambda^\mathsf{e}$.  The relative Hochschild homology $H_*(\Lambda |C,X)$ of a $\Lambda$-bimodule $X$ is defined as $\Tor_*^{\Lambda^\mathsf{e} |C^\mathsf{e}}(X, \Lambda).$ Note that G. Hochschild considered the extension $ C\otimes \Lambda^{\mathsf{op}}\subset \Lambda^\mathsf{e}$ in \cite{HOCHSCHILD1956}, but the $\Tor$ vector spaces obtained this way are the same, see for instance \cite{CLMS Pacific}.

If $C=k$,  relative Hochschild homology is Hochschild homology as defined in \cite{HOCHSCHILD1945}, which is denoted $H_*(\Lambda ,X)$. If $X=\Lambda$ the usual notation is $HH_*(\Lambda)$.

\vskip5mm
The setting of \cite{CLMS2021nearlyexactJZ} is as follows.

 \begin{defi}\label{sequence nearly exact}
A sequence of positively graded chain complexes of vector spaces  \begin{equation}\label{seq}
      0\to C_*\stackrel{\iota}{\to}  D_*\stackrel{\kappa}{\to} E_*\to 0
      \end{equation} with $\iota$ injective, $\kappa$ surjective and $\kappa\iota=0$ is called \textit{short nearly exact}. The middle quotient complex $\left({\Ker\kappa}/{\Im \iota}\right)_*$ is called the \textit{gap complex}.

 \end{defi}

\begin{rema}\label{double complex}

Consider the double complex with columns at $p=0,1$ and $2$ given by the short nearly exact sequence (\ref{seq}) after the usual change of signs. Recall that the horizontal maps go from right to left, since the double complex is of homological type.

By filtering the double complex by rows we have a spectral sequence. At page $1$ the only possible non zero column is column $1$, which is the gap complex $(Ker \kappa/Im \iota)_*$. Therefore the spectral sequence converges to $H_*\left({\Ker\kappa}/{\Im \iota}\right)$.
\end{rema}

\begin{defi}\label{definition long nearly exact}
A \textit{long nearly exact sequence}  is a complex of vector spaces
\begin{align*}
 \dots &\stackrel{\delta}{\to} U_{m} \stackrel{I}{\to} V_{m} \stackrel{K}{\to} W_{m} \stackrel{\delta}{\to} U_{m-1} \stackrel{I}{\to} V_{m-1}\to \dots \\&\stackrel{\delta}{\to} U_{n} \stackrel{I}{\to} V_{n} \stackrel{K}{\to} W_{n}
\end{align*}
ending at some $n$ which is exact at $W_m$ for all $m>n$ and at all $U_m$. The graded vector space $(Ker K/ Im I)_*$ is called the \textit{gap} of the long nearly exact sequence.
\end{defi}

Next we adjust and we extend Theorem 4.2 of \cite{CLMS2021nearlyexactJZ}. Part a) is new, while part b) requires the additional hypothesis (\ref{missing}) which is missing in \cite{CLMS2021nearlyexactJZ}. It is worth noting that Jonathan Lindell wrote to us raising a question about the results in \cite{CLMS2021nearlyexactJZ}, just after we realized that Theorem 4.2 of \cite{CLMS2021nearlyexactJZ} needs this hypothesis.

\begin{theo}\label{theo gap}
Consider a short nearly exact sequence of positively graded chain complexes
 \begin{equation}\label{towards JZ nes}0\to C_*\stackrel{\iota}{\to}  D_*\stackrel{\kappa}{\to} E_*\to 0.\end{equation}
 \renewcommand{\labelenumi}{\alph{enumi})}
 \begin{enumerate}
 \item If $H_*(\Ker \kappa /\Im \iota)=0$ for $*>>0$, then there is long exact sequence
 \begin{align}\label{long n exact sequence}\dots &\stackrel{\delta}{\to} H_{m}(C_*) \stackrel{I}{\to} H_{m}(D_*) \stackrel{K}{\to} H_{m}(E_*) \stackrel{\delta}{\to} H_{m-1}(C_*)  \stackrel{I}{\to} \dots  \nonumber
 \\&\dots \stackrel{\delta}{\to} H_{n}(C_*) \stackrel{I}{\to} H_{n}(D_*) \stackrel{K}{\to} H_{n}(E_*)
\end{align}
ending at some $n$.

\item Let $I$ and $K$ be the maps induced in homology by $\iota$ and $\kappa$ respectively. If
 \begin{equation}\label{missing}\dim_k(\Ker K/ \Im I)_*= \dim_kH_*(\Ker \kappa/ \Im\iota) \mbox{ for } *>>0\end{equation} then the sequence  (\ref{long n exact sequence}) exists and it is a long nearly exact sequence.
\end{enumerate}

\end{theo}

\begin{proof}
We now filter the double complex of Remark \ref{double complex} by columns. At page $1$ the vertical differentials are $0$. The horizontal ones are
$$ 0 \leftarrow H_m(E_*) \stackrel{K}{\leftarrow} H_{m}(D_*)\stackrel{I}{\leftarrow} H_{m}(C_*)\leftarrow 0$$

At page $2$ we have the following:

\begin{itemize}
\item  column $0$ is $(\Coker K)_*$,
\item  column $1$ is $(\Ker K/\Im I)_*$,
\item  column $2$ is $(\Ker I)_*$
\end{itemize}
and all the  other columns are zero.

At page $2$, the differentials are zero except possibly
 $$(E^2_{0, q+1}=\Coker K) \stackrel{d_2}{\longleftarrow}(E^2_{2,q}=\Ker I)$$
see for instance \cite[p. 122]{WEIBEL}.   Hence at page $3$ we have that:
 \begin{itemize}
\item  column $0$ is $(\Coker d_2)_*$,
\item  column $1$ is $(\Ker K/\Im I)_*$,
\item  column $2$ is $(\Ker d_2)_*$,
\end{itemize}
all the  other columns are zero and the differentials are zero. Hence these vector spaces remain the same in the following pages. Consequently the spectral sequence converges to $$(\Coker d_2)_{*+1} \oplus (\Ker K/\Im I)_* \oplus(\Ker d_2)_*.$$

Both filtrations converge to the same limit, hence Remark \ref{double complex} gives
\begin{align}\label{equality dimensions}
\dim_k(\Ker d_2)_* + \dim_k(\Ker K/\Im I)_* + \dim_k(\Coker d_2)_{*+1} =\nonumber \\ \dim_kH_*(\Ker \kappa/ \Im\iota).&
\end{align}

For a) we assume that $H_*(\Ker \kappa /\Im \iota)=0$ for $*>>0$, hence
$$\dim_k(\Ker d_2)_* + \dim_k(\Ker K/\Im I)_* + \dim_k(\Coker d_2)_{*+1}=0 \mbox{ for } *>>0.$$
In particular we have  $(\Ker K /\Im I)_*=0$ for $*>>0$. In other words the sequence is exact at $H_m(D_*)$ for $m>>0.$

Moreover, in high enough degrees we have $$\dim_k(\Ker d_2)_* =0 = \dim_k(\Coker d_2)_{*+1}$$
which means that $d_2$ is invertible in high enough degrees.

Let us consider the canonical maps $p: H_m(E_*) \twoheadrightarrow \Coker K$ and $q: \Ker I \hookrightarrow H_{m-1}(C_*)$. Define
$$\delta= q \left(d_2^{-1}\right) p.$$ We do have $\Ker\delta=\Im K$ and $\Im \delta= \Ker I$.

For b), the hypothesis (\ref{missing}) implies that $d_2$ is invertible in high enough degrees. The previous construction provides $\delta$, which  gives a long nearly exact sequence which gap is $H_*(\Ker \kappa/ \Im\iota)$.
\qed
\end{proof}

In the following, we confirm the results in \cite{CLMS2021nearlyexactJZ} after Theorem \ref{theo gap}.

Let $C\subset \Lambda$ be an extension of algebras. Let $X$ be a $\Lambda$-bimodule. By \cite[Theorem 3.3]{CLMS2021nearlyexactJZ} there is a \textit{fundamental nearly exact sequence} for $*>1$ \begin{equation}\label{funda}
0\to \check{C}_*(C,X)\stackrel{\iota}{\to}  \check{C}_*(\Lambda,X)\stackrel{\kappa}{\to} \check{C}_*(\Lambda|C,X)\to 0
\end{equation}
where the positively graded complexes are defined in  \cite[Section 2]{CLMS2021nearlyexactJZ}. The homology of these complexes gives respectively $H_{*}(C,X)$,  $H_{*}(\Lambda,X)$ and  $H_{*}(\Lambda|C,X).$

We denote by $I$ and $K$ the maps in Hochschild homology induced  by $\iota$ and $\kappa$ respectively.
In this context, we reformulate below Theorem 4.4 of \cite{CLMS2021nearlyexactJZ}, which is actually a particular case of the above Theorem \ref{theo gap}.
\begin{theo}\label{JZ}
Let $C\subset \Lambda$ be an extension of algebras and let $X$ be a $\Lambda$-bimodule. With the above notations, we have
 \renewcommand{\labelenumi}{\alph{enumi})}
\begin{enumerate}
 \item If $H_*(\Ker \kappa /\Im \iota)=0$ for $*>>0$, then there is a Jacobi-Zariski sequence
 \begin{align}\label{JJZZ}\dots &\stackrel{\delta}{\to} H_{m}(C,X) \stackrel{I}{\to} H_{m}(\Lambda,X) \stackrel{K}{\to} H_{m}(\Lambda|C,X) \stackrel{\delta}{\to}H_{m-1}(C,X) \stackrel{I}{\to}\dots\nonumber
\\\tag{JZ} &\stackrel{\delta}{\to} H_{n}(C,X) \stackrel{I}{\to} H_{n}(\Lambda,X) \stackrel{K}{\to} H_{n}(\Lambda|C,X)
\end{align}
which is long exact and ends for some $n$.
\item If $\dim_k(\Ker K/ \Im I)_*= \dim_kH_*(\Ker \kappa/ \Im\iota)$ for  $*>>0,$ then the sequence  (\ref{JJZZ}) exists and it is a long nearly exact sequence.
\end{enumerate}

\end{theo}

The name Jacobi-Zariski is given after \cite[p. 61]{ANDRE},\cite{Iyengar}, see also \cite{CLMS2021nearlyexactJZ}.

\medskip
Next we approximate the homology of the gap $H_*(\Ker \kappa/ \Im \iota)$ of the fundamental sequence (\ref{funda}). This way we reconsider  Theorem 5.1 of \cite{CLMS2021nearlyexactJZ}.

\begin{theo}\label{homology of JZ}
With the above notations, if $\Tor_*^C(\Lambda/C, (\Lambda/C)^{\otimes_C n})=0$ for $*>0$ and for all $n$,
then there is a spectral sequence converging to $H_*(\Ker \kappa/ \Im \iota)$  in large enough degrees. Its terms at page $1$ are
$$E^1_{p,q}= \Tor^{C^\mathsf{e}}_{q}(X,(\Lambda/C)^{\otimes_Cp}) \mbox{   \ \ for } p,q>0$$
and $0$ anywhere else.

\end{theo}
\begin{proof}
Theorem 5.1 of  \cite{CLMS2021nearlyexactJZ} intends to approximate the gap of the Jacobi-Zariski sequence (\ref{JJZZ}). The proof there  focuses on the homology of the gap complex. This focus is now our aim.

Thus the proof of \cite[p. 1645, Theorem 5.1]{CLMS2021nearlyexactJZ} is relevant avoiding its first three lines. \qed
\end{proof}

In the following, we confirm that the previous tools provide an alternative proof of the results of A. Kaygun in \cite{KAYGUN} as  in  \cite[Theorem 6.2]{CLMS2021nearlyexactJZ}.

\begin{theo}
Let $C\subset \Lambda$ be an extension of $k$-algebras such that $\Lambda/C$ is a flat $C$-bimodule, and let $X$ be a $\Lambda$-bimodule. There is a Jacobi-Zariski long exact sequence
\begin{align*} \dots &\stackrel{\delta}{\to} H_{m}(C,X) \stackrel{I}{\to} H_{m}(\Lambda,X) \stackrel{K}{\to} H_{m}(\Lambda|C,X) \stackrel{\delta}{\to}H_{m-1}(C,X) \stackrel{I}{\to}\dots\\
& \stackrel{\delta}{\to} H_{n}(C,X) \stackrel{I}{\to} H_{n}(\Lambda,X) \stackrel{K}{\to} H_{n}(\Lambda|C,X)
\end{align*}
 ending at some $n$.
\end{theo}
\begin{proof}
$\Lambda/C$ is flat as a left and as a right $C$-module (see for instance the first part of the proof of \cite[Lemma 6.1]{CLMS2021nearlyexactJZ}), hence  $\Tor_*^C(\Lambda/C, (\Lambda/C)^{\otimes_C n})=0$ for $*>0$ and for all $n$.

Now by Theorem \ref{homology of JZ}, there is a spectral sequence converging to the homology of the gap of the fundamental sequence (\ref{funda}) in large enough degrees. The first page of this spectral sequence is
$$E^1_{p,q}= \Tor^{C^\mathsf{e}}_{q}(X,(\Lambda/C)^{\otimes_Cp}) \mbox{   \ \ for } p,q>0$$
and $0$ elsewhere. For $p>0$ we have that the $C^\mathsf{e}$-module $(\Lambda/C)^{\otimes_C p}$ is flat, see for instance \cite[Lemma 6.1]{CLMS2021nearlyexactJZ}. If $q>0$, then $\Tor^{C^\mathsf{e}}_{q}(X,(\Lambda/C)^{\otimes_Cp})=0$. Consequently the first page of the spectral sequence is $0$, so the homology of the gap of the fundamental sequence is $0$. Then by  Theorem \ref{theo gap} a), there exists a long Jacobi-Zariski exact sequence as stated. \qed

\end{proof}
We recall from \cite{CLMS2021nearlyexactJZ,CLMS2022bounded} that an extension of algebras $C\subset \Lambda$  is \textit{left (respectively right) bounded} if
\begin{itemize}
\item
$\Lambda/C$ is projective as a  left (respectively right) $C$-module,
\item
$\Lambda/C$  is tensor nilpotent as a $C$-bimodule,
\item
$\Lambda/C$ is of finite projective dimension as a $C$-bimodule.
\end{itemize}

We confirm now \cite[Theorem 6.5]{CLMS2021nearlyexactJZ} - see also \cite[Theorem 2.9]{CLMS2022bounded}, by means of the previous  results. We underline that we consider an extension of algebras $C\subset \Lambda$ which is not necessarily split, namely it may not exist a two sided ideal $I$ of $\Lambda$ such that $\Lambda=C\oplus I$.

\begin{theo}
With the above notations, assume that the extension is left or right bounded.
Then there is a Jacobi-Zariski long exact sequence
\begin{align*} \dots &\stackrel{\delta}{\to} H_{m}(C,X) \stackrel{I}{\to} H_{m}(\Lambda,X) \stackrel{K}{\to} H_{m}(\Lambda|C,X) \stackrel{\delta}{\to}H_{m-1}(C,X) \stackrel{I}{\to}\dots \\&\stackrel{\delta}{\to} H_{n}(C,X) \stackrel{I}{\to} H_{n}(\Lambda,X) \stackrel{K}{\to} H_{n}(\Lambda|C,X)
\end{align*}
ending at some $n$.
\end{theo}
\begin{proof}
We have that $\Tor_*^C(\Lambda/C, (\Lambda/C)^{\otimes_C n})=0$ for $*>0$ and for all $n$. Hence by Theorem \ref{homology of JZ} there is a spectral sequence converging to $H_*(\Ker \kappa/ \Im \iota)$  in large enough degrees.  At page $1$ we have $E^1_{p,q}= \Tor^{C^\mathsf{e}}_{q}(X,(\Lambda/C)^{\otimes_Cp})$ for $p,q>0$ and $0$ otherwise. Let $u$ be the projective dimension of the $C$-bimodule $\Lambda/C$. Then  $(\Lambda/C)^{\otimes_C p}$ is of projective dimension at most $pu$, see \cite[Chapter IX, Proposition 2.6]{CARTANEILENBERG}.

Let $v$ be such that $(\Lambda/C)^{\otimes_Cv}=0$. Note that if $p\geq v$ or $q>pu$, then $E^1_{p,q}=0$. Therefore if $p+q\geq v(u+1)$, then $E^1_{p,q}=0$. That is the terms of the spectral sequence vanish at page $1$ for high enough total degrees. Hence $H_*(\Ker \kappa/ \Im \iota)=0$ for $*>>0$. Then Theorem \ref{theo gap} part a) provides the Jacobi-Zariski long exact sequence.

\qed
\end{proof}

\begin{rema} Consider a Morita context $\Lambda=\begin{pmatrix}
                                              A  & N \\
                                              M & B
                                            \end{pmatrix}_{\alpha, \beta}$ and let $C=A\times B$ be its diagonal subalgebra.
\begin{itemize}
  \item It may happen that  $(\Lambda,\begin{pmatrix}
                                              1  & 0 \\
                                              0& 0
                                            \end{pmatrix})$  is strongly stratifying but $C\subset \Lambda$ is not a bounded extension. Nevertheless in this case $\Lambda/C$ is tensor nilpotent, as we will see in the next section.
  \item If the  extension $C\subset \Lambda$ is bounded, then $M\otimes_A N=0$ if and only if the Morita context is strongly stratifying.
\end{itemize}

\end{rema}

\section{\sf Hochschild homology of strongly stratifying Morita contexts}\label{HH of strongly str MC}

The main purpose of this section is to prove that if a strongly stratifying Morita context \small $\Lambda=\begin{pmatrix}
                                              A  & N \\
                                              M & B
                                            \end{pmatrix}_{\alpha, \beta}$ \normalsize has finite Hochschild homology, then the same holds for its diagonal subalgebra $C=A\times B$ and consequently for each diagonal algebra $A$ and $B$.

We first recall some easy to show facts that we will use.

\begin{enumerate}
\renewcommand*\labelenumi{(F\theenumi)}
  \item  A left $C$-module $X$ is the direct sum of a left $A$-module ${}_aX$ and a left $B$-module ${}_bX$, where ${}_aX=(1,0)X$   and  ${}_bX=(0,1)X$.

      Conversely, if ${}_aX$ and ${}_bX$ are left $A$ and $B$-modules respectively, then ${}_aX\oplus {}_bX$ is a left $C$-module. Note that $B$ and $A$ annihilate respectively $_aX$, and  $_bX=0$.
  \item In particular a left $A$-module $U$ becomes a left $A\times B$-module through $U\oplus 0$, with $BU=0$.
  \item\label{zero tensor} Let $U$ (resp. $V$) be a right $A$ (resp. left $B$)-module, viewed as a right (resp. left) $C$-module. We have $U\otimes_C V=0.$
\item \label{tor Atimes B}
Let $Y$ (resp. $X$) be a right (resp. left) $C$-module and $Y=Y_a\oplus Y_b$ (resp. $X= {}_aX \oplus {}_bX$) be the decomposition as above. We have $$\Tor^{C}_*(Y_a\oplus Y_b, {}_aX \oplus {}_bX)= \Tor^A_* (Y_a,{}_aX)\oplus \Tor^B_*(Y_b,{}_bX).$$
Indeed a left $C$-projective resolution of ${}_aX \oplus {}_bX$ is given by the direct sum of a left $A$-projective resolution of ${}_aX$ and a left $B$-projective resolution of  ${}_bX$. Using (F\ref{zero tensor}) we infer the result.

\item\label{decomposition as Cbimodule}  As $C=A\times B$,
$$C^\mathsf{e}=
A^\mathsf{e} \times
(A\otimes B^{\mathsf{op}} ) \times
(B\otimes A^{\mathsf{op}})\times
B^\mathsf{e}.$$
Let $X$ be a $C$-bimodule. We have
\begin{align*}
X= &{}_aX_a \oplus {}_aX_b\ \oplus {}_bX_a\oplus{}_bX_b
\end{align*}
where  ${}_aX_a$,  ${}_aX_b$, ${}_bX_a$ and ${}_bX_b$ are respectively an $A$-bimodule, an $A-B$-bimodule, a $B-A$-bimodule and a $B$-bimodule.

\item \label{tor Cbimodules} Let $Y$ and $X$ be $C$-bimodules decomposed as above. After (F\ref{tor Atimes B}) we have
\begin{align*}
\Tor^{C^\mathsf{e}}_*(Y,X)= &\Tor_*^{A^\mathsf{e}}({}_aY_a,{}_aX_a)\oplus \Tor_*^{A\otimes B^{\mathsf{op}}}({}_bY_a,{}_aX_b)\oplus \\ &\Tor_*^{B\otimes A^{\mathsf{op}}}({}_aY_b,{}_bX_a)\oplus \Tor_*^{B^\mathsf{e}}({}_bY_b,{}_bX_b).
\end{align*}
\item\label{M oplus N} As mentioned the $B-A$-bimodule $M$ is viewed as a $C$-bimodule by extending the actions by zero, that is $AM=MB=0$. Analogously, $N$ is a $C$-bimodule.
This way $$\Lambda/C= M\oplus N$$ as $C$-bimodules.
\end{enumerate}

Next we show that the hypotheses of Theorem \ref{homology of JZ} hold  for a strongly stratifying Morita context.

\begin{prop}\label{preparation}
Let $\Lambda$ be a strongly stratifying Morita context $\begin{pmatrix}
                                              A  & N \\
                                              M & B
                                            \end{pmatrix}_{\alpha, \beta}$, and let $C=A\times B$ as a subalgebra of the Morita context.

We have that $\Tor_*^C(\Lambda/C, (\Lambda/C)^{\otimes_C n})=0$ for $*>0$ and for all $n$.
\end{prop}
\begin{proof}
As noted in  (F\ref{M oplus N}), $\Lambda/C= M\oplus N$ as $C$-bimodules. To compute $$(M\oplus N)^{\otimes_C2}$$ note that $M\otimes_CM =0=N\otimes_CN$ by (F\ref{zero tensor}). Analogously, $N\otimes_C M= N\otimes_B M$ and $M\otimes_C N= M\otimes_A N$; the latter is $0$ since the Morita context is strongly stratifying. Finally
\begin{equation}\label{2explicit}(M\oplus N)^{\otimes_C2} = N\otimes_B M.\end{equation}
Moreover
\begin{equation}\label{3 zero}
(M\oplus N)^{\otimes_C3}= M\otimes_A N\otimes_B M =0.
\end{equation}

For $n\geq 3$, we infer $(M\oplus N)^{\otimes_Cn}=0$ and $\Tor_*^C(\Lambda/C, (\Lambda/C)^{\otimes_C n})=0.$

For $n=2$ we have
\begin{align*}\Tor_*^C(M\oplus N, (M\oplus N)^{\otimes_C2} )&= \Tor_*^C(M\oplus N, N\otimes_B M )\\&= \Tor_*^A(M, N\otimes_B M )\\
&\stackrel{\star}{=}\Tor_*^B(M\otimes_AN, M) \\
&=0
\end{align*}
The equality $\stackrel{\star}{=}$ is ensured by \cite[Theorem 2.8, p.167]{CARTANEILENBERG} in case the following takes place
$$\Tor_n^A(M,N)=0=\Tor_n^B(N,M) \mbox{ \ \ \ for }n>0.$$ Indeed, this holds since the Morita context is strongly stratifying.

For $n=1$
 we have
\begin{align*}
  \Tor_*^C(\Lambda/C, (\Lambda/C))&=\Tor_*^C(M\oplus N, M\oplus N)\\
&= \Tor_*^A(M, N)\oplus \Tor_*^B(N, M)
\end{align*}
according to (F\ref{tor Atimes B}).
Now $\Tor_*^A(M, N)=0$ for $*>0$ since the Morita context is stratifying. Moreover $\Tor_*^B(N, M)=0$ for $*>0$ since it is strongly stratifying.
\qed
\end{proof}
We will now show that the terms at the first page of the spectral sequence of Theorem \ref{homology of JZ} for $X=\Lambda$ vanishes.

\begin{lemm}\label{CEzero}
Let $\Lambda =\begin{pmatrix}
                                              A  & N \\
                                              M & B
                                            \end{pmatrix}_{\alpha, \beta}$ be a strongly stratifying Morita context. For $n\geq 0$
                                           $$\Tor_n^{A\otimes B^{\mathsf{ op}}} (M,N)=0= \Tor_n^{B\otimes A^{\mathsf{op}}}(N,M).$$
\end{lemm}
\begin{proof}
We make use of the ``associativity formula" of H. Cartan and S. Eilenberg \cite[p. 347, (5a)]{CARTANEILENBERG}, namely there is a spectral sequence
$$H_q(B, \Tor_p^A(M,N))\Rightarrow \Tor_n^{A\otimes B^{\mathsf{ op}}} (M,N).$$
Since the Morita context is strongly stratifying, $\Tor_p^A(M,N)=0$ for $p\geq 0$. Hence $H_q(B, \Tor_p^A(M,N))=0$ for all $p$ and $q$, and $\Tor_n^{A\otimes B^{\mathsf{ op}}} (M,N)=0$ for all $n$.

Given an algebra $D$, a right $D$-module $X$ and a left $D$-module $Y$, it is well known that for all $n$
$$\Tor_n^D(X,Y)=\Tor_n^{D^{\mathsf{op}}}(Y,X).$$
Hence
$$\Tor_q^{B\otimes A^{\mathsf{op}}}(N,M)=\Tor_q^{A\otimes B^{\mathsf{ op}}} (M,N)=0.$$\qed
\end{proof}

\begin{prop}\label{first page vanish}
 Let $\Lambda$ be a strongly stratifying Morita context $\begin{pmatrix}
                                              A  & N \\
                                              M & B
                                            \end{pmatrix}_{\alpha, \beta}$, and let $C=A\times B.$
We have
 $$ \Tor^{C^\mathsf{e}}_{q}(\Lambda,(\Lambda/C)^{\otimes_Cp})=0 \mbox{   \ \ for } p,q>0.$$
\end{prop}
\begin{proof}
By (\ref{3 zero}) we have $(\Lambda/C)^{\otimes_Cp}=0$ for $p\geq 3$, thus
$\Tor^{C^\mathsf{e}}_{q}(\Lambda,(\Lambda/C)^{\otimes_Cp})=0$ for $p\geq 3$.

The decomposition of (F\ref{decomposition as Cbimodule}) of $\Lambda$ as a $C$-bimodule is $$\Lambda = A\oplus N \oplus M\oplus B.$$

For $p=1$, according to (F\ref{tor Cbimodules}) we have
$$\Tor^{C^\mathsf{e}}_{q}(A\oplus N \oplus M\oplus B,\  M\oplus N)=\Tor_q^{B\otimes A^{\mathsf{op}}}(N,M) \oplus \Tor_q^{A\otimes B^{\mathsf{ op}}} (M,N).$$
Both summands vanish by the previous Lemma \ref{CEzero}.

For $p=2$, we know by (\ref{2explicit}) that $(M\oplus N)^{\otimes_C 2}=N\otimes_B M$. Hence
\begin{align*}
\Tor^{C^\mathsf{e}}_{q}(A\oplus N \oplus M\oplus B,\ (M\oplus N)^{\otimes_C 2} )= \Tor^{A^\mathsf{e}}_{q}(A, N\otimes_B M).
\end{align*}
We show next that the hypotheses of \cite[p.347 (4a)]{CARTANEILENBERG} hold: firstly note that $\Tor_n^A(A,N)=0$ for $n>0$. Secondly, since the Morita context is strongly stratifying, $\Tor_n^B(N,M)=0$ for $n>0$. Therefore there is an isomorphism
$$\Tor_q^{A^\mathsf{e}}(A, N\otimes_BM) \simeq \Tor_q^{A^\mathsf{op}\otimes B}(A\otimes_A N, M).$$
The latter is $\Tor_q^{A^\mathsf{op}\otimes B}(N, M),$ which is zero by the previous Lemma \ref{CEzero}.\qed
\end{proof}

\begin{theo}\label{JZ esmc*}
 Let $\Lambda$ be an algebra with a distinguished idempotent $e$, such that $\Lambda e \Lambda$ is a strongly stratifying ideal  and let $C= e\Lambda e \times f\Lambda f$, where $f=1-e$.

There exists a Jacobi-Zariski long exact sequence
\begin{align*} \dots &\stackrel{\delta}{\to} H_{m}(C,\Lambda) \stackrel{I}{\to} H_{m}(\Lambda,\Lambda) \stackrel{K}{\to} H_{m}(\Lambda|C,\Lambda) \stackrel{\delta}{\to}H_{m-1}(C,\Lambda) \stackrel{I}{\to}\dots\\
& \stackrel{\delta}{\to} H_{n}(C,\Lambda) \stackrel{I}{\to} H_{n}(\Lambda,\Lambda) \stackrel{K}{\to} H_{n}(\Lambda|C,\Lambda)
\end{align*}
 ending at some $n$.
\end{theo}
\begin{proof}
Consider the corresponding strongly stratifying Morita context \\ $\Lambda=\begin{pmatrix}
                                              A  & N \\
                                              M & B
                                            \end{pmatrix}_{\alpha, \beta}$
where $A=e\Lambda e$ and $B=f\Lambda f$, thus $C=A\times B$. We will show that we can use part  a) of Theorem \ref{JZ}, that is we assert  $H_*(\Ker \kappa /\Im \iota)=0$ for $*>>0$. Indeed, the spectral sequence of Theorem \ref{homology of JZ} is available if we prove that $\Tor_*^C(\Lambda/C, (\Lambda/C)^{\otimes_C n})=0$ for $*>0$ and for all $n$. This follows from Proposition \ref{preparation}.

Moreover the first page of this spectral sequence  is
$$E^1_{p,q}= \Tor^{C^\mathsf{e}}_{q}(\Lambda,(\Lambda/C)^{\otimes_Cp}) \mbox{   \ \ for } p,q>0$$
and $0$ anywhere else, see Theorem \ref{homology of JZ}  for $X=\Lambda$.

This first page vanishes, due to Proposition \ref{first page vanish}. We have proved that $$H_*(\Ker \kappa /\Im \iota)=0 \mbox{ for } *>>0$$ and of Theorem \ref{JZ} a) provides the existence of the Jacobi-Zariski long exact sequence.\qed
\end{proof}

\begin{lemm}\label{0 para m mayor que 2}
In the situation of Theorem \ref{JZ esmc*}, let $X$ be a $\Lambda$-bimodule. We have that $H_{m}(\Lambda|C, X)=0$ for $m\geq 3$.
\end{lemm}
\begin{proof}
By \cite[Corollary 2.4]{CLMS2021nearlyexactJZ}  we have that $H_{m}(\Lambda|C,X)$ is the homology of the following chain complex
 \begin{equation*}\ \ \cdots  \stackrel{}{\to} X\otimes_{C^\mathsf{e}} (\Lambda/C)^{\otimes_Cm}\stackrel{}{\to}\cdots \stackrel{}{\to} X\otimes_{C^\mathsf{e}}\Lambda/C \stackrel{}{\to} X_C\to 0\end{equation*}
  where $X_C= \Lambda\otimes_{C^\mathsf{e}}C=X/\langle cx-xc\rangle = H_0(C,X).$
  On the other hand, (\ref{3 zero}) ensures that $(\Lambda/C)^{\otimes_C m}=0$ for $m\geq 3$.\qed

\end{proof}

\begin{theo}\label{key}
Let $\Lambda$ be an algebra with a distinguished idempotent $e$ such that $\Lambda e \Lambda$ is a strongly stratifying ideal, and let $f=1-e$.  If $HH_*(\Lambda)$ is finite, then $HH_*(e\Lambda e)$ and $HH_*(f \Lambda f)$ are finite.
\end{theo}
\begin{proof}
Consider the corresponding strongly stratifying Morita context \\ $\Lambda=\begin{pmatrix}
                                              A  & N \\
                                              M & B
                                            \end{pmatrix}_{\alpha, \beta}$ where $A=e\Lambda e$ and $B=f\Lambda f$ and $C=A\times B$. The  Lemma \ref{0 para m mayor que 2} and the Jacobi-Zariski long exact sequence of Theorem \ref{JZ esmc*} provide $H_*(C, \Lambda)= 0$ for $*>>0$. Moreover
$$\Tor_*^{C^\mathsf{e}}(C, \Lambda)= \Tor_*^{C^\mathsf{e}}(A\oplus B, \ A\oplus N\oplus M\oplus B).$$
By (F\ref{tor Cbimodules}) the latter is $$\Tor^{A^\mathsf{e}}(A,A)\oplus \Tor^{B^\mathsf{e}}(B,B)=HH_*(A)\oplus HH_*(B).$$ We infer that $HH_*(A)$ and $HH_*(B)$ are finite.\qed
\end{proof}

\section{\sf Han's conjecture}\label{HAN}

We recall Han's conjecture \cite{HAN2006} for an algebra $\Lambda$: if $HH_*(\Lambda)$ is finite, then $\Lambda$ has finite global dimension.

\begin{theo}\label{diagonales OK entonces el contexto OK}
Let $\Lambda$ be an algebra with a distinguished idempotent $e$ such that $\Lambda e \Lambda$ is a strongly stratifying ideal,  and let $f=1-e$.
The algebra $\Lambda$ verifies Han's conjecture if and only if $e\Lambda e\times f\Lambda f$ does.
\end{theo}

\begin{proof}
Consider the corresponding strongly stratifying Morita context $$\Lambda = \begin{pmatrix}
                                              A  & N \\
                                              M & B
                                            \end{pmatrix}_{\alpha, \beta},$$ where $A=e\Lambda e$ and $B=f\Lambda f$.
Assume that $A\times B$ satisfies Han's conjecture and let us prove that this is also the case for $\Lambda$. So suppose $HH_*(\Lambda)$ is finite.  By  Theorem \ref{key},  we have that $HH_*(A)$ and $HH_*(B)$ are finite. It is well known that $HH_*(A\times B) = HH_*(A)\oplus HH_*(B)$, then $H_*(A\times B)$ is finite and thus $A\times B$ has finite global dimension. Hence $A$ and $B$ have finite global dimension.

Note that by Remark \ref{esmc is smc} we have $\beta:0\to B$, then $\Im\beta=0$ and thus $B/\Im\beta=B$. Therefore $\Lambda/ \Lambda e \Lambda =B$, see Remark \ref{Lambda modulo the ideal through e}.    Since the  ideal $\Lambda e \Lambda$ is strongly stratifying, it is stratifying. Hence there is a recollement of $D(\Lambda)$ relative to $D(\Lambda/\Lambda e \Lambda)$ and $D(e\Lambda e)$, that is relative to $D(B)$ and $D(A)$.

According to L. Angeleri H\"{u}gel, S. Koenig, Q. Liu and D. Yang in \cite[Theorem I, p. 17]{ANGELERI KOENIG LIU YANG 2017 472} (see also \cite[Proposition 4, p.541]{HAN2014}), since $A$ and $B$ have finite global dimension, $\Lambda$ has finite global dimension.

Next we show the other implication. Assume that $\Lambda$ satisfies Han's conjecture, our aim is to show that $A\times B$ also does. Let's suppose that  $HH_*(A\times B)$ is finite. Since $HH_*(A\times B) = HH_*(A)\oplus HH_*(B)$ we infer that $HH_*(A)$ and $HH_*(B)$ are finite.

 We have that $\Lambda e\Lambda$ is a strongly stratifying ideal, hence it is stratifying and there is a recollement. According to \cite[Corollary 2, p. 543]{HAN2014} after B. Keller \cite{KELLER1998}, there is a long exact sequence in Hochschild homology
\begin{equation*}
  \cdots\to HH_{n+1}(\Lambda/\Lambda e \Lambda) \to HH_n(e\Lambda e)\to HH_n(\Lambda) \to HH_n(\Lambda/\Lambda e \Lambda) \to \cdots
\end{equation*}
that is for the Morita context
\begin{equation*}
  \cdots\to HH_{n+1}(B) \to HH_n(A)\to HH_n(\Lambda) \to HH_n(B) \to \cdots .
\end{equation*}
We infer that $HH_*(\Lambda)$ is finite. Since $\Lambda$ verifies Han's conjecture,  $\Lambda$ is of finite global dimension.

Using again the above cited result in \cite{ANGELERI KOENIG LIU YANG 2017 472,HAN2014} we infer that $A$ and $B$, and thus $A\times B$, have finite global dimension. \qed
\end{proof}
Theorem \ref{diagonales OK entonces el contexto OK} will also be useful for considering algebras filtered by ideals which successive quotients provide strongly stratifying ideals. The following result shows that  Definition \ref{chain strongly stratifying} below makes sense.

\begin{lemm}\label{equality}
Let $\Lambda$ be an algebra and let $u,v\in \Lambda$ be orthogonal idempotents. We have
$$\frac{\Lambda(u+v)\Lambda}{\Lambda u \Lambda} = \frac{\Lambda}{\Lambda u \Lambda}\overline{v}\frac{\Lambda}{\Lambda u \Lambda}$$ where  $\overline{v}$ is the class of $v$ in ${\Lambda}/{\Lambda u \Lambda}$.

\end{lemm}
\begin{proof}
First note that since $u$ and $v$ are orthogonal idempotents we have $$\Lambda(u+v)\Lambda = \Lambda u \Lambda + \Lambda v \Lambda,$$ consequently
$$\frac
{\Lambda(u+v)\Lambda}{\Lambda u \Lambda}=\frac{\Lambda u \Lambda + \Lambda v \Lambda} {\Lambda u \Lambda}=
 \frac{\Lambda v\Lambda}{\Lambda u\Lambda \cap \Lambda v\Lambda}\cdot$$
Next  consider the composition $\Lambda v \Lambda \hookrightarrow \Lambda \twoheadrightarrow
{\Lambda}/{\Lambda u \Lambda}.$ Its image is  $\frac
{\Lambda}{\Lambda u \Lambda}\overline{v}\frac
{\Lambda}{\Lambda u \Lambda}$ and its kernel is $\Lambda u\Lambda \cap \Lambda v\Lambda$.\qed
\end{proof}

\begin{defi}\label{chain strongly stratifying}
Let $\Lambda$ be an algebra. A \textit{strongly stratifying $n$-chain} is an ordered complete system of orthogonal idempotents $\{e_1,\dots,e_n\}$ of $\Lambda$ such that the filtration by ideals
$$0\subset \Lambda e_1\Lambda \subset \Lambda (e_1+e_2)\Lambda \subset \cdots \subset \Lambda (e_1+e_2+\dots + e_{n-1})\Lambda \subset \Lambda$$
verifies that for  $1\leq i\leq n$ the quotient $ {\Lambda (e_1+\dots + e_{i})\Lambda}/{\Lambda (e_1+\dots + e_{i-1})\Lambda} $
is a strongly stratifying ideal of ${\Lambda}/{\Lambda (e_1+\dots + e_{i-1})\Lambda}.$
\end{defi}
\begin{rema}\
\begin{itemize}
 \item The bound quiver algebra $\Lambda$ of Example \ref{LIU}
\footnotesize
\[\xymatrix{&&e_2\ar@<.7ex>[dd]^{c}\ar@<-.7ex>[dd]_{b}\\
e_1\ar[rru]^{a}&&&\\
&& e_3\ar[llu]^{d}}
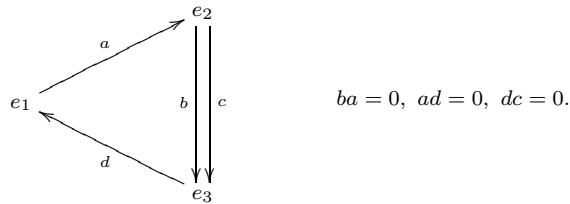
\hspace{10pt}\xymatrix{\\ ba=0,~ ad=0,~dc=0.}\]
\normalsize admits a strongly stratifying $2$-chain $\{e_2 +e_3, e_1\}.$
 \item For $1\leq i\leq n$, let $\Lambda_i={\Lambda}/{\Lambda (e_1+\dots + e_{i})\Lambda}.$ According to Lemma \ref{equality}
$$\frac {{\Lambda (e_1+\dots + e_{i})\Lambda}}{{\Lambda (e_1+\dots + e_{i-1})\Lambda}} = \Lambda_{i-1} \overline{e_i} \Lambda_{i-1}$$
where $\overline{e_i}$ denotes the class of $e_i$ in $\Lambda_{i-1}.$
\end{itemize}
\end{rema}

\begin{defi}
Let $\C$ be a class of algebras. A \textit{$\C$-strongly stratifying $n$-chain} of an algebra $\Lambda$ is a strongly stratifying $n$-chain $\{e_1,\dots,e_n\}$ of $\Lambda$ such that for $1\leq i\leq n$ the algebra $e_i\Lambda e_i$ belongs to $\C$.
\end{defi}

We will need the following lemma.

\begin{lemm}\label{preparation for Han for chains}
  Let $\C$ be a class of algebras which is closed by taking quotients.   Let $\Lambda$ be an algebra admitting a $\C$-strongly stratifying $n$-chain $\{e_1,\dots,e_n\}$ for $n>1$. The algebra $\Lambda/\Lambda e_1 \Lambda$ admits a  $\C$-strongly stratifying $n-1$-chain $\{\overline{e_2},\overline{e_3},\dots,\overline{e_n}\}.$

\end{lemm}
\begin{proof} We have the following quotient filtration of $\Lambda/\Lambda e_1 \Lambda$
$$0\subset  \frac{\Lambda (e_1+e_2)\Lambda}{\Lambda e_1 \Lambda} \subset \cdots \subset \frac{\Lambda (e_1+e_2+\dots + e_{n-1})\Lambda }{\Lambda e_1 \Lambda} \subset \frac{\Lambda}{\Lambda e_1 \Lambda}\cdot$$
Using Lemma \ref{equality},  the ideals of this filtration are as follows
$$ \frac{\Lambda(e_1+\dots +e_i)\Lambda}{\Lambda e_1 \Lambda}= \frac{\Lambda(e_2+\dots +e_i)\Lambda}{\Lambda e_1 \Lambda \cap \Lambda(e_2+\dots +e_i)\Lambda}= \frac{\Lambda}{\Lambda e_1 \Lambda} (\overline{e_2}+\dots + \overline{e_i}) \frac{\Lambda}{\Lambda e_1 \Lambda}\cdot$$
Hence the quotient filtration is indeed the one corresponding to the complete system of orthogonal idempotents $\{\overline{e_2},\dots ,\overline{e_i}\}$ of $\Lambda/\Lambda e_1 \Lambda$, namely
$$0\subset  \frac{\Lambda}{\Lambda e_1 \Lambda}\overline{e_2} \frac{\Lambda}{\Lambda e_1 \Lambda}\subset \cdots \subset \frac{\Lambda}{\Lambda e_1 \Lambda} (\overline{e_2}+\dots + \overline{e_n}) \frac{\Lambda}{\Lambda e_1 \Lambda}\subset \frac{\Lambda}{\Lambda e_1 \Lambda}\cdot$$
To verify that the successive quotients of this filtration of $\Lambda/\Lambda e_1 \Lambda$ are strongly stratifying in the corresponding quotient of $\Lambda/\Lambda e_1 \Lambda$, note that
$$\frac {{\Lambda (e_1+\dots + e_{i})\Lambda}/\Lambda e_1 \Lambda}{{\Lambda (e_1+\dots + e_{i-1})\Lambda}/\Lambda e_1 \Lambda}=\frac {{\Lambda (e_1+\dots + e_{i})\Lambda}}{{\Lambda (e_1+\dots + e_{i-1})\Lambda}}$$
and
$$\frac{{\Lambda/\Lambda e_1 \Lambda}}{{\Lambda (e_1+\dots + e_{i-1})\Lambda}/\Lambda e_1 \Lambda}=\frac{{\Lambda}}{{\Lambda (e_1+\dots + e_{i-1})\Lambda}}\cdot$$

Finally observe that
$$\overline{e_i}\frac{\Lambda}{\Lambda e_1\Lambda}\overline{e_i}=\frac{e_i\Lambda e_i}{\Lambda e_1 \Lambda \cap e_i\Lambda e_i}.$$
Since $e_i\Lambda e_i$ is in $\C$ which is closed by taking quotients, we infer that $\overline{e_i}\frac{\Lambda}{\Lambda e_1\Lambda}\overline{e_i}$ also belongs to $\C$. This way $\{\overline{e_2},\dots ,\overline{e_n}\}$ is a  $\C$-strongly stratifying $n-1$-chain of $\Lambda/\Lambda e_1 \Lambda$.
 \qed
\end{proof}
\begin{theo}\label{Han for chains}
 Let $\C$ be a class of algebras which is closed by taking quotients, and assume that Han's conjecture holds for all algebras in $\C$.   Let $\Lambda$ be an algebra which admits a $\C$-strongly stratifying $n$-chain for some $n>0$. Then $\Lambda$ verifies Han's conjecture.
\end{theo}
\begin{proof} By induction, let $\Lambda$ be an algebra admitting a $\C$-strongly stratifying $n$-chain $\{e_1,\dots, e_n\}.$    If $n=1$, then $e_1=1$ and the algebra $\Lambda = 1\Lambda 1$ is in $\C$. By hypothesis, $\Lambda$ verifies Han's conjecture.

Let $n>1$ and consider the algebra $\Lambda/\Lambda e_1 \Lambda$ which admits a $\C$-strongly stratifying $n-1$-chain by  Lemma \ref{preparation for Han for chains}. Hence Han's conjecture holds for it.

The ideal $\Lambda e_1 \Lambda$ is strongly stratifying in $\Lambda$. By Theorem \ref{diagonales OK entonces el contexto OK}, Han's conjecture holds for $\Lambda$ if and only if it holds  for $e_1\Lambda e_1\times (1-e_1)\Lambda (1-e_1)$. To verify the latter, suppose that $HH_*(e_1\Lambda e_1\times (1-e_1)\Lambda (1-e_1))$ is finite.
  So $HH_*(e_1\Lambda e_1)$ is finite. But $e_1\Lambda e_1$  belongs to $\C$, thus by hypothesis it verifies Han's conjecture. Then $e_1\Lambda e_1$ has finite global dimension.

On the other hand we also have that $HH_*((1-e_1)\Lambda (1-e_1))$ is finite. Consider the Morita context given by $e_1$. By Remark \ref{Lambda modulo the ideal through e} $$ \frac{\Lambda}{\Lambda e_1 \Lambda} = \frac{(1-e_1)\Lambda (1-e_1)}{\Im\beta}\cdot$$
Since $\Lambda e_1 \Lambda$ is strongly stratifying, we have that $\beta=0$. Consequently
$$ \Lambda/\Lambda e_1 \Lambda = (1-e_1)\Lambda (1-e_1)$$
and $HH_*(\Lambda/\Lambda e_1 \Lambda)$ is finite. Han's conjecture holds for $\Lambda/\Lambda e_1 \Lambda$ by the inductive hypothesis - see above.  Then $\Lambda/\Lambda e_1 \Lambda=(1-e_1)\Lambda (1-e_1)$ has  finite global dimension.

We have established before that $e_1\Lambda e_1$ has finite global dimension. We infer that $e_1\Lambda e_1\times (1-e_1)\Lambda (1-e_1)$ has finite global dimension, that is this algebra verifies Han's conjecture as needed.  \qed
\end{proof}

\begin{coro}
Assume that Han's conjecture holds for local algebras. If an algebra $\Lambda$ admits a strongly stratifying chain $\{e_1,\dots, e_n\}$ with  $e_i$ primitive for all $i$, then Han's conjecture is true for $\Lambda$.
\end{coro}

In order to avoid classes of algebras closed by taking quotients, instead of filtering by ideals of an algebra $\Lambda$,  below we filter $\la$ by algebras $f\Lambda f$, where $f$ is a partial sum of a complete system of orthogonal idempotents.

The following lemma can be easily proved.
\begin{lemm}\label{trivial}
Let $\Lambda$ be an algebra with an ordered complete set of orthogonal idempotents $\{e_1,\dots, e_n\}$. Consider the following idempotents
$$f_0=e_{1}+\dots +e_n,\ f_1= e_{2}+\dots +e_n,\ \dots\ ,f_i= e_{i+1}+\dots +e_n, \ \dots ,\ f_{n-1}= e_n.$$

For  $0\leq i \leq n-1$,  consider the algebra $f_i\Lambda f_i$ with unit $f_i$.
\begin{itemize}
  \item For $j> i$ we have $e_j=f_ie_j=e_jf_i=f_ie_jf_i$, therefore $e_j\in f_i\Lambda f_i$ and $e_j(f_i\Lambda f_i )e_j = e_j\Lambda e_j$,
 \item $f_i\Lambda f_i$ has a complete set of orthogonal idempotents $\{e_{i+1},\dots, e_n\}$ and $f_{i+1}=f_i - e_{i+1}$,
   \item For $j\geq i$ we have $f_i f_{j}= f_{j}f_{i}= f_{j}$, therefore $f_{j}\la f_{j}=f_{j}(f_i\Lambda f_i )f_{j}$,
   \item
$0\ \subset \ f_{n-1}\la f_{n-1}\ \subset\ \dots\ \subset\  f_{i}\la f_{i}\ \subset\ \dots \ \subset\ f_{1}\la f_{1} \ \subset f_{0}\la f_{0}=\Lambda.$
\end{itemize}
\end{lemm}

We keep the notations of Lemma \ref{trivial} in the sequel.
\begin{defi}
Let $\Lambda$ be an algebra. A \textit{strongly co-stratifying $n$-chain} of $\Lambda$ is an ordered complete set of  orthogonal idempotents $\{e_1,\dots, e_n\}\subset \Lambda$
such that the ideals provided by the following idempotents
$$f_1 \in \Lambda,\ f_2\in f_{1}\la f_{1},\ \dots\ ,\ f_{i+1}\in f_i\la f_i,\  \dots\ , \ f_{n-1}\in f_{n-2}\la f_{n-2}$$
 are strongly stratifying in their respective algebras.
\end{defi}

\begin{exam}
The bound quiver algebra $\Lambda$ of Example \ref{LIU}
\footnotesize
\[\xymatrix{&&e_2\ar@<.7ex>[dd]^{c}\ar@<-.7ex>[dd]_{b}\\
e_1\ar[rru]^{a}&&&\\
&& e_3\ar[llu]^{d}}
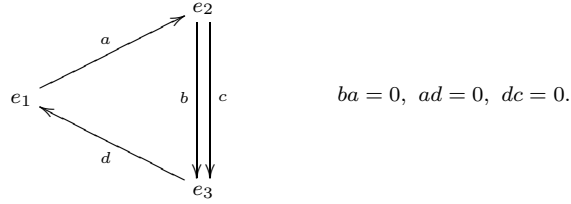
\hspace{10pt}\xymatrix{\\ ba=0,~ ad=0,~dc=0.}\]
\normalsize
admits  a  strongly co-stratifying $3$-chain $\{e_1,e_2,e_3\}$. Indeed, consider the filtration
$$0\ \subset \  e_3\Lambda e_3 \ \subset\   (e_2 + e_3)\Lambda (e_2 + e_3)\ \subset\ \Lambda .$$
We know that the idempotent $e_2+e_3$ provides a strongly stratifying ideal in $\Lambda$. Moreover $e_3$ gives trivially a strongly stratifying ideal of the Kronecker algebra  \newline $(e_2 + e_3)\Lambda (e_2 + e_3)$ since the corresponding Morita context is $\begin{pmatrix}
                                              k  &  k\oplus k\\
0 & k
                                            \end{pmatrix}$.
\end{exam}
\begin{defi}
 An \textit{$\H$-strongly co-stratifying $n$-chain} of $\Lambda$ is a strongly co-stratifying $n$-chain $\{e_1,\dots, e_n\}$  such that $e_i\Lambda e_i$ verifies Han's conjecture for all $i$.
\end{defi}
\begin{theo}
Let $\Lambda$ be an algebra which admits an $\H$-strongly co-stratifying $n$-chain.  Then $\Lambda$ verifies Han's conjecture.
\end{theo}
\begin{proof}  By induction,  let $\Lambda$ be an algebra which admits an $\H$-strongly co-stratifying $n$-chain $\{e_1,\dots, e_n\}$. If $n=1$, then $e_1=1$  and $\la=e_1\Lambda e_1$ verifies Han's conjecture.

For $n>1$, recall that $f_1=e_2+\dots + e_n=1-e_1$.  Since $\Lambda f_1 \Lambda $ is a strongly stratifying ideal of $\Lambda$, by  Theorem \ref{diagonales OK entonces el contexto OK} we have that Han's conjecture holds for $\Lambda$ if and only if it holds  for $f_1\la f_1\times e_1\Lambda e_1$. To verify the latter, suppose that $HH_*( f_1\la f_1\times e_1\Lambda e_1)$ is finite, then $HH_*(f_1\la f_1)$ and $HH_*(e_1\Lambda e_1)$ are finite.  We have that  $e_1\Lambda e_1$  verifies Han's conjecture    thus $e_1\Lambda e_1$ is of finite global dimension.

On the other hand   we assert that $f_1\la f_1$ admits a  $\H$-strongly co-stratifying $(n-1)$-chain $\{e_2,\dots, e_n\}$. First by Lemma \ref{trivial}, for $j\geq 2$ we have $f_j(f_1\la f_1) f_j = f_j\la f_j$. Thus the ideal provided by $f_{j+1}$ in $f_j(f_1\la f_1) f_j$ is strongly stratifying since $\{e_1,\dots, e_n\}$ is a strongly co-stratifying $n$-chain of $\la$. Second, by Lemma \ref{trivial} for $j\geq 2$, we have $e_j(f_1\la f_1) e_j = e_j\la e_j$ and the latter verifies Han's conjecture.

Therefore the inductive hypothesis ensures that $f_1\la f_1$ verifies Han's conjecture,   hence $f_1\la f_1$ is of finite global dimension. We infer that $f_1\la f_1\times e_1\Lambda e_1$ is of finite global dimension, that is Han's conjecture is true for $f_1\la f_1\times e_1\Lambda e_1$.
\qed
\end{proof}

\begin{coro}
  Assume that Han's conjecture holds for local algebras. If an algebra $\la$ admits a co-stratifying chain consisting of primitive idempotents, then Han's conjecture is true for $\la$.
\end{coro}

\begin{rema}
 As quoted in the Introduction, a comparison between algebras admitting a strongly stratifying or co-stratifying chain with algebras which are standardly stratified will be considered in a forthcoming paper. For the convenience of the reader, we recall the definition of standardly stratified algebras (see for instance \cite{AGOSTON DLAB LUKACS, AGOSTON HAPPEL LUKACS UNGER,  MARCOS MENDOZA SAENZ SANTIAGO, XI}).

 With the same notations as in Lemma \ref{trivial}, recall that  $f_i= e_{i+1}+\dots +e_n$, and let $f_n=0$. Consider the set $\Delta$ of \textit{standard left $\Lambda$-modules} $\Delta_i = \Lambda e_i /   \Lambda f_i\Lambda  e_i $ for $i=1,\dots, n$. As mentioned in \cite{AGOSTON DLAB LUKACS} the module $\Delta_i$ is the largest quotient of $\Lambda e_i$ such that its composition factors are not isomorphic to $(\Lambda / r)e_j$ for $j>i$, where $r$ is the radical of $\Lambda$.

 The algebra $\Lambda$ is \textit{standardly stratified} if it admits a filtration by left submodules which successive quotients belong to $\Delta$, up to isomorphism.
\end{rema}

\section{\sf Patterns for examples of strongly stratifying Morita contexts}
 In the following we provide patterns for obtaining families of strongly stratifying Morita contexts, through assuming  projectivity hypothesis for $M$ and/or $N$.

 \begin{rema}\
 \begin{itemize}
   \item In Example \ref{LIU} from \cite[Example 4.4]{LIU VITORIA YANG},\cite[Example 2.3]{ANGELERI KOENIG LIU YANG}, neither $M$ is projective as a right $A$-module, nor $N$ is projective as a left $A$-module.
   \item In \cite{CIBILSREDONDOSOLOTAR} Morita contexts with $\alpha=\beta=0$ and $M$ and $N$ projective bimodules are considered. In what follows, in general $\alpha\neq 0$. In Proposition \ref{any N}, $N$ is any bimodule and $M$ is a projective bimodule. In Proposition  \ref{Mleftproj Nleftproj}, $M$ and $N$ are left projective modules.
       \item We emphasize that our results for a strongly stratifying Morita context do not depend on the morphism $\alpha$, while $\beta=0$ since its source vector space vanishes. In other words changing $\alpha$ to $\alpha'$ provides in general different Morita contexts, nevertheless the Morita context remains strongly stratifying and the results of this paper still apply.
 \end{itemize}
\end{rema}

\begin{lemm}\label{associative conditions with beta zero}

Let $\Lambda$ be a Morita context  $\begin{pmatrix}
                                              A  & N \\
                                              M & B
                                            \end{pmatrix}_{\alpha, \beta}$ with $\beta=0$.
The associativity conditions \ref{associativity conditions} are equivalent to
$$(\Im \alpha) N=0= M(\Im \alpha). $$

\end{lemm}

\begin{prop}\label{MandN bimodules proyectives}
Let $A$ and $B$ be algebras, $a$, $a'$ be idempotents in $A$ and $b$, $b'$ be idempotents in $B$. Let
$$M=Bb\otimes aA \mbox{ and } N= Aa'\otimes b'B.$$
Let $\alpha: N\otimes_BM\to A$ be a morphism of $A$-bimodules.
There is a  strongly stratifying Morita context $\begin{pmatrix}
                                              A  & N \\
                                              M & B
                                            \end{pmatrix}_{\alpha, \beta}$ if and only if $aAa'=0.$
\end{prop}
\begin{proof}
Note that $M$ and $N$ are projective bimodules, so they are left and right projective. Thus both $\Tor_n^A(M,N)=0$ and
 $\Tor_n^B(N,M)=0$ are zero for $n>0$.
 Also
 $$M\otimes_A N = Bb\otimes aA \otimes_A Aa'\otimes b'B= Bb\otimes aAa'\otimes b'B.$$
 If the Morita context is strongly stratifying then $M\otimes_A N=0$. We infer $aAa'=0$.
 Conversely, if $aAa'=0$, then $M\otimes_A N=0$.   Note that  $N\otimes_B M= Aa'\otimes b'Bb \otimes aA$, so $\Im \alpha\subset Aa'AaA.$ Consequently
 $$(\Im \alpha) N \subset Aa'AaAa'\otimes b'B =0 \mbox{ and } M(\Im \alpha) \subset Bb\otimes aAa'AaA=0.$$
 The associativity conditions of Lemma \ref{associative conditions with beta zero} are satisfied, thus there is a Morita context. \qed
\end{proof}
\begin{rema}
Under the hypothesis of Proposition \ref{MandN bimodules proyectives}
\begin{align*}
\dim_k \Hom_{A-A} (N\otimes_B M, A)= &\dim_k \Hom_{A-A} (Aa'\otimes b'Bb \otimes aA, A)\\=&\dim_k (a'Aa) \dim_k (b'Bb).
\end{align*}
Hence it is possible to choose $\alpha\neq 0$ if and only if $a'Aa\neq 0$ and $b'Bb\neq 0$.
\end{rema}

In the following we provide an example for Proposition \ref{MandN bimodules proyectives}, keeping the same notations.

\begin{exam}
Let $A$ be the algebra of the quiver
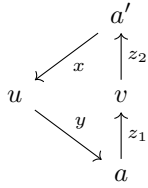
\[\begin{tikzcd}
	& {a'} \\
	u & v \\
	& a
	\arrow["x", from=1-2, to=2-1]
	\arrow["y", from=2-1, to=3-2]
	\arrow["{z_1}"', from=3-2, to=2-2]
	\arrow["{z_2}"', from=2-2, to=1-2]
\end{tikzcd}\]
with the relation $yx=0$. Let $B=k$, with $b=b'=1$. We have $M=aA$ and $N=Aa'$. Moreover, $aAa'=0.$
The projective $A$-bimodule $N\otimes_B M$ is $Aa'\otimes aA$. We have
$$\Hom_{A-A}(Aa'\otimes aA, A)=a'Aa = k\{z_2z_1\}.$$
A non-zero $\alpha$ is determined by $\alpha (a'\otimes a)=z_2z_1$. We denote by $m$ and $n$ the generators $a$ and $a'$ of $M$ and $N$ respectively. The strongly stratifying Morita context has a presentation given by the quiver
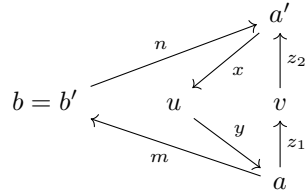
\[\begin{tikzcd}
	&& {a'} \\
	{b=b'} & u & v \\
	&& a
	\arrow["x", from=1-3, to=2-2]
	\arrow["y", from=2-2, to=3-3]
	\arrow["n", from=2-1, to=1-3]
	\arrow["m", from=3-3, to=2-1]
	\arrow["{z_1}"', from=3-3, to=2-3]
	\arrow["{z_2}"', from=2-3, to=1-3]
\end{tikzcd}\]
and the relations $yx=0$ and $nm=z_2z_1$.

\end{exam}

\begin{exam}
Consider $A$ the algebra of the quiver
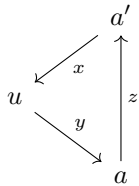
\[\begin{tikzcd}
	& {a'} \\
	u \\
	& a
	\arrow["x", from=1-2, to=2-1]
	\arrow["y", from=2-1, to=3-2]
	\arrow["z"', from=3-2, to=1-2]
\end{tikzcd}\]

with the relation $yx=0$. Let $B=k$, with $b=b'=1$. We have $M=aA$ and $N=Aa'$. Moreover, $aAa'=0.$
The projective $A$-bimodule $N\otimes_B M$ is $Aa'\otimes aA$. We have
$$\Hom_{A-A}(Aa'\otimes aA, A)=a'Aa = k{z}.$$
A non-zero $\alpha$ is determined by $\alpha (a'\otimes a)=z$. We denote $m$ and $n$ the generators of $M$ and $N$ respectively. The strongly stratifying Morita context has a presentation given by the quiver
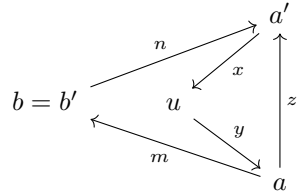
\[\begin{tikzcd}
	&& {a'} \\
	{b=b'} & u \\
	&& a
	\arrow["x", from=1-3, to=2-2]
	\arrow["y", from=2-2, to=3-3]
	\arrow["z"', from=3-3, to=1-3]
	\arrow["n", from=2-1, to=1-3]
	\arrow["m", from=3-3, to=2-1]
\end{tikzcd}\]
and the relations $yx=0$ and $nm=z$. An admissible presentation of this Morita context is given by the quiver
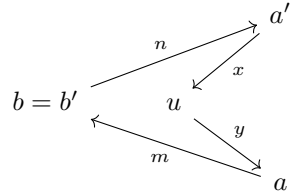
\[\begin{tikzcd}
	&& {a'} \\
	{b=b'} & u \\
	&& a
	\arrow["x", from=1-3, to=2-2]
	\arrow["y", from=2-2, to=3-3]
	\arrow["n", from=2-1, to=1-3]
	\arrow["m", from=3-3, to=2-1]
\end{tikzcd}\]
and the admissible relation $yx=0$.

In other words, this algebra is also a Morita context but relative to an algebra $A'$ instead of $A$.

\end{exam}

The projectivity requirements for $M$ and $N$ can be relaxed as follows.

\begin{prop}\label{any N}
Let $A$ and $B$ be algebras with respective idempotents $a$ and $b$. Let
$M=Bb\otimes aA$, and let $N$ be any $A-B$-bimodule. Let $\alpha: N\otimes_B M\to A$ be an $A$-bimodule map. There is a  strongly stratifying Morita context $\begin{pmatrix}
                                              A  & N \\
                                              M & B
                                            \end{pmatrix}_{\alpha, \beta}$ if and only if $aN=0$ and $a(\Im\alpha)=0.$
\end{prop}

\begin{proof}
For $n>0$ we have  $\Tor_n^A(M,N)=0$ and $\Tor_n^B(N,M)=0$. Moreover $M\otimes_A N= Bb\otimes aN$, hence $M\otimes_A N=0$ if and only if $aN=0$.
Also, $N\otimes_B M= Nb\otimes aA$, hence $\Im\alpha \subset Aa$.

If the Morita context is strongly stratifying then $M\otimes_A N= 0$, whence $aN=0$.
By Lemma \ref{associative conditions with beta zero}, $M(\Im\alpha)=0$, that is $Bb\otimes Aa(\Im\alpha )=0$ which is equivalent to $a(\Im\alpha) =0$.

For the converse, it remains to prove that $(\Im\alpha) N=0$ in order to satisfy the conditions of Lemma \ref{associative conditions with beta zero}. We have that $\Im\alpha \subset Aa$. Hence
$$\Im\alpha N\subset AaN =0.$$\qed
\end{proof}
\begin{prop}\label{Mleftproj Nleftproj}
Let $A$ and $B$ be algebras with respective idempotents $a$ and $b$. Let $M'\neq 0$ be a right $A$-module and
$M=Bb\otimes M'$. Let $N'\neq 0$ be a right $B$-module and $N= Aa\otimes N'$. Let $\alpha: N\otimes_B M\to A$ be an $A$-bimodule map. There is a  strongly stratifying Morita context $\begin{pmatrix}
                                              A  & N \\
                                              M & B
                                            \end{pmatrix}_{\alpha, \beta}$ if and only if $M'a=0$ and $(\Im\alpha)a=0.$
\end{prop}

\begin{proof}
We have that
\begin{itemize}
  \item $\Tor_n^A(M,N)=0$ and $\Tor_n^B(N,M)=0$ for $n>0$,
  \item $M\otimes_A N= Bb\otimes M'a \otimes N'$, whence $M\otimes_A N=0$ if and only if $M'a=0$,
  \item $N\otimes_B M= Aa\otimes N'b \otimes M'$, whence $\Im \alpha \subset aA$.
\end{itemize}

If the Morita context is strongly stratifying then $M'a=0$.
By Lemma \ref{associative conditions with beta zero}, $(\Im\alpha)N=0$, that is $(\Im\alpha)Aa\otimes N'=0$ which is equivalent to $ (\Im\alpha)a=0$.

For the converse, it remains to prove that $M(\Im\alpha)=0$. We have that $(\Im\alpha) \subset aA$. Hence
$$M(\Im\alpha) = Bb\otimes M' (\Im\alpha) \subset Bb\otimes M'aA =0.$$\qed

\end{proof}

\footnotesize
\noindent C.C.:\\
Institut Montpelli\'{e}rain Alexander Grothendieck, CNRS, Univ. Montpellier, France.\\
{\tt Claude.Cibils@umontpellier.fr}

\medskip
\noindent M.L.:\\
Instituto de Matem\'atica y Estad\'\i stica  ``Rafael Laguardia'', Facultad de Ingenier\'\i a, Universidad de la Rep\'ublica, Uruguay.\\
{\tt marclan@fing.edu.uy}

\medskip
\noindent E.N.M.:\\
Departamento de Matem\'atica, IME-USP, Universidade de S\~ao Paulo, Brazil.\\
{\tt enmarcos@ime.usp.br}

\medskip
\noindent A.S.:
\\IMAS-CONICET y Departamento de Matem\'atica,
Facultad de Ciencias Exactas y Naturales,\\
Universidad de Buenos Aires, Argentina. \\{\tt asolotar@dm.uba.ar}

\end{document}